\documentclass{scrartcl}

\usepackage{url}

\usepackage{mathtools}  
\usepackage{amsmath}
\usepackage{amsthm}
\usepackage{amsxtra}

\usepackage{amssymb}

\usepackage{amscd}

\usepackage[all]{xy}

\usepackage[english]{hyperref}
\usepackage[english]{babel}
\usepackage[utf8]{inputenc}

\newcommand*{\field}[1]{\mathbb{#1}}
\newcommand*{\Q}{\field{Q}}
\newcommand*{\F}{\field{F}}
\newcommand*{\R}{\field{R}}
\newcommand*{\C}{\field{C}}
\newcommand*{\Z}{\field{Z}} 

\newcommand*{\OF}{\mathcal{O}_\F}
\newcommand*{\DF}{\mathcal{D}_\F}

\newcommand*{\group}[1]{\mathrm{#1}}

\newcommand*{\SU}{\group{SU}}
\newcommand*{\Ug}{\group{U}}
\newcommand*{\Og}{\group{O}}
\newcommand*{\SO}{\group{SO}}
\newcommand*{\SL}{\group{SL}}

\newcommand*{\VF}{V_{\F}}
\newcommand*{\VQ}{V_{\Q}}
\newcommand*{\VFC}{V_{\F}(\C)} 
\newcommand*{\VQR}{V_{\Q}(\R)} 
\newcommand*{\VQC}{V_{\Q}(\C)} 
\newcommand*{\PFR}{\mathbb{P}^1{(V_{\F})(\C)}} 
\newcommand*{\PQC}{\mathbb{P}^1{(V_{\Q})(\C)}}  
\newcommand*{\piU}{\pi_{1}} 
\newcommand*{\piO}{\pi_{2}} 

\newcommand*{\GrO}{\mathrm{Gr}_\Og} 

\newcommand*{\HO}{\mathcal{H}_\Og}
\newcommand*{\Hp}{\mathbb{H}}

\newcommand*{\coneU}{\mathcal{K}_\Ug}
\newcommand*{\coneO}{\mathcal{K}_\Og} 

\newcommand*{\posQuad}{\mathcal{C}_{+}}

\DeclarePairedDelimiter{\hlfa}{\langle}{\rangle}
\newcommand*{\hlf}[2]{\hlfa*{ #1, #2}} 
\newcommand{\hlfempty}{\hlf{\cdot}{\cdot}}
\DeclarePairedDelimiter{\blfp}{(}{)}
\newcommand*{\blf}[2]{\blfp*{ #1, #2}} 
\newcommand*{\blfempty}{\blf{\cdot}{\cdot}} 
\newcommand*{\Qf}[1]{q\!\left( #1 \right)} 
\newcommand*{\QfNop}{q} 
\DeclarePairedDelimiter{\abs}{\lvert}{\rvert}

\DeclareMathOperator{\tr}{Tr}
\DeclareMathOperator{\norm}{N}

\newcommand*{\Mweak}{\mathcal{M}^!} 

\DeclareMathOperator{\Div}{div}
\newcommand*{\HeegO}{H}
\newcommand*{\HeegU}{\mathbf{H}}
\newtheorem{theorem}{Theorem}[section]

\newtheorem{remark}{Remark}[theorem]

\newtheorem{corollary}{Corollary}[theorem]
\numberwithin{equation}{section}

\title{Borcherds Products for $\Ug(1,1)$}
 \author{Eric Hofmann
\footnote{%
 Eric Hofmann. Mathematisches Institut Universität Heidelberg,\newline
Im Neuenheimer Feld 288, 69120 Heidelberg, \newline 
{Email: 
hofmann@mathi.uni-heidelberg.de} 
} }
\date{April 25, 2013}

\begin{document}
\maketitle

\section*{Abstract}
In \cite{Ho12} the author constructed a multiplicative Borcherds lift for indefinite unitary groups $\Ug(1,n)$. In the present paper the case of $\Ug(1,1)$ is examined in greater detail. 
The lifting in this case takes weakly holomorphic elliptic modular forms of weight zero as inputs and lifts them to meromorphic modular forms for $\Ug(1,1)$, on the usual complex upper half-plane $\Hp$. In this setting, the Weyl-chambers can be described very explicitly and the associated Weyl-vectors can also be calculated. This is carried out in detail for weakly holomorphic modular forms with $q$-expansions of the form $q^{-n} + \mathcal{O}(q)$, $n>0$, and for a constant function, which together span the input space, $\Mweak_0(\Gamma)$. 
The general case for the lifting of an arbitrary $f \in \Mweak_0(\Gamma)$ comes by as a corollary. 
The lifted functions take their zeros and poles along Heegner divisors, which consist of CM-points in $\Hp$. We find that their
 CM-order can to some extent be prescribed. \\
{2010 \textit{Mathematics Subject Classification:} 
11F27, 11F41, 11F55, 11G18, 14G35.}\\
{\textit{Key words and phrases:}  Borcherds product, unitary modular form,
Heegner divisor, CM point, unitary modular variety}

\section{Introduction and statement of results}

In his seminal paper \cite{Bo98}, Borcherds constructed a multiplicative lifting
from weakly holomorphic modular forms transforming  under a Weil-representation of
$\SL_2(\Z)$  to meromorphic automorphic forms for indefinite orthogonal groups
of signature $(2,p)$, $p\geq 2$, which have infinite product expansion and whose zeros
and poles lie along certain arithmetic cycles, called Heegner-divisors. 
In \cite{Ho11}, \cite{Ho12} this construction is transferred to unitary groups of
signature $(1,n)$, $n\geq 1$, via an embedding of $\Ug(1,n)$ into the orthogonal group
$\Og(2,2n)$. The Borcherds products in \cite{Ho12}, just as in \cite{Bo98}, come with a 
term involving a Weyl-vector, which often remains difficult to determine in practice.
It relates to a Lorentzian quadratic subspace of the underlying rational space.
Borcherds' lifting has been generalized by Bruinier in \cite{Br02}. If $f$ is a 
weakly holomorphic modular form, then, one of Bruinier's results implies that the
Weyl-chambers and Weyl-vectors related to $f$ have decompositions ruled
by the principal part in the Fourier expansion of $f$ (see \cite{Br02}, p.\ 88).

In the present paper, we focus on the case $\Ug(1,1)$. 
The results of Bruinier afford a method to calculate explicitly the Weyl-chambers 
and Weyl-vectors involved.
Note that the additive lift to $\SO(1,1)$, which determines the Weyl-vector in
this case is not entirely covered in \cite{Br02}.
In contrast, the main result of \cite{Ho12}, theorem 8.1, does apply to $\Ug(1,1)$,  
however, some useful simplifications can be made in this case. 

 The input functions here are weakly holomorphic modular forms of weight
zero for the elliptic modular group $\Gamma = \SL_2(\Z)$.
 While Bruinier uses non-holomorphic Poincar\'{e} series as inputs, 
 we consider the lift of the uniquely defined weakly holomorphic modular
 forms $j_n$ with $q$-expansion $q^{-n} + \mathcal{O}(q)$, for $n\in\Z_{>0}$.
 These forms (together with a constant)
 span the space $\Mweak_0(\Gamma)$. 
We introduce non-holomorphic Poincar\'{e} series $F_m$, for $m\in\Z$, $m<0$, in 
section \ref{sec:poincarejn} and recall the relationship between the families $F_m$ and $j_n$. 

In sections \ref{sec:LandHp} and \ref{sec:LandHO} we briefly describe the setting
for the later sections and for the application of the results from \cite{Ho12},
sections 4 and 8:
First, in section \ref{sec:LandHp} we recall the construction for the symmetric domain of
the unitary group. For $\Ug(1,1)$, it can be realized as the usual complex
upper-half plane $\Hp = \{ \tau\in\C\,;\, \Im\tau>0\}$, as follows.
Let $\F = \Q(\sqrt{d})$ be an imaginary quadratic number field, $\VF$ a two-dimensional
hermitian space over $\F$ with a non-degenerate indefinite hermitian form
$\hlfempty$, and $L$ a hermitian lattice in $\VF$, given by $L =
\OF\oplus\DF^{-1}$, where $\OF$ is the ring of integers in $\VF$ and $\DF^{-1}$
the inverse different ideal. This lattice, $L$, is even and unimodular. 

The symmetric domain for $\Ug(V)(\R)\simeq \Ug(1,1)$ 
is isomorphic to the cone of positive definite one-dimensional subspaces in 
$\mathbb{P}^{1}(\C)$, denoted $\coneU$.
We take a fixed choice of two lattice vectors of zero norm $\ell$ and $\ell'$ with
 $\hlf{\ell}{\ell'} \neq 0$ and $\ell$ primitive.
Then, $\Hp$ is identified with $\coneU$ by assigning to each $\tau \in \Hp$ a unique representative
$z = \ell' - \tau\delta\hlf{\ell'}{\ell}\ell$, with $[z] \in \coneU$.

Note also that the special unitary group $\SU(1,1)$ is isomorphic to $\SL_2(\R)$, and,
in particular, for the arithmetic subgroup $\SU(L)$, we have  $\SU(L)\simeq\SL_2(\Z)$. 

In section \ref{sec:LandHO}, we recall the construction of the tube domain model $\HO$ of the symmetric domain for the indefinite special orthogonal group $\SO(2,2)$, since Borcherds' lifting takes inputs from $\Mweak_0(\Gamma)$ to automorphic forms for this group. In the present case, $\HO$ can be identified with a product of two upper half-planes $\HO \simeq \Hp \times \Hp$.
In this section, we consider $L$ as a $\Z$-module with
the bilinear form $\blfempty \vcentcolon = \tr_{\F/\Q} \hlfempty$. We denote by $\VQ$ 
the $\Q$-vector space underlying $\VF$, as a rational quadratic space of signature $(2,2)$, 
with $\blfempty$ extended to a $\Q$-valued bilinear form.
Also, in this section, we introduce the Lorentzian $\Z$-sublattice $K \simeq \Z^2$, which will be studied further in sections
\ref{sec:PhimK} -- \ref{sec:correction}.  

This identification between  $\VQ$ and the rational space underlying $\VF$ 
results in an embedding of $\Ug(V)$ into $\SO(V)$, described in section
\ref{sec:embedHp}. It represents a special case of the
embeddings studied in \cite{Ho11} chapter 3 and \cite{Ho12} section 4. In the
present situation, through the identification $\HO \simeq \Hp \times \Hp$,
the embedding can be written simply as
$\Hp \hookrightarrow
\Hp\times\Hp$, $\tau \mapsto (\tau, -\bar\zeta)$, where $\zeta$ is a generator
of $\OF$ as a $\Z$-module, either $\sqrt{d}$ or $\frac{1 + \sqrt{d}}{2}$, 
depending on whether the discriminant of $\F$ is even or odd. 
(Here, $\sqrt{d}$ is always defined using the principal branch of
 the complex square root.)

The Weyl-vectors we want to calculate arise as follows: To each input function $f\in\Mweak_0(\Gamma)$ one associates connected components of the symmetric domain for the special orthogonal group $\SO(K\otimes_\Z\R)$ of the Lorentzian space $K\otimes_\Z\R$, called Weyl-chambers.
Note that the symmetric domain can be realized as a quadrant $\posQuad = \{ Y = (y_1, y_2)\,;\, y_1, y_2 > 0\}$ of the $Y$-plane $K\otimes_\Z\R$. For each Weyl-chamber $W$ the (additive) lifting to $\SO(K\otimes_\Z\R)$, 
denoted $\Phi^K(Y; f)$, has an expansion adapted to this Weyl-chamber, $\Phi^K(Y; f,W)$. 
The Weyl-vector $\rho(f;W) \in K\otimes_\Z\R$ is then  defined through
\[
\Phi^K(Y; f, W) = 8\sqrt{2}\pi \blf{\frac{Y}{\abs{Y}}}{ \rho(f;W)}. 
\]     
In section \ref{sec:PhimK}, starting with an integral expression from \cite{Br02}, for the lift of a non-holomorphic Poincar\'{e} series, we obtain an expansion for $\Phi_m^K(Y) \vcentcolon = \Phi^K(Y; F_m)$. In the following section \ref{sec:weylchambers_K}, we examine the Weyl-chambers for $F_m$. They can easily be listed using the divisors of $\abs{m}$, as they take the form
\[
 W(t_i, t_{i+1}) = \left\{ Y \in \posQuad\,;\, \frac{t_i^2}{\abs{m}} y_2 < y_1 < \frac{t_{i+1}^2}{\abs{m}} y_2 \right\},
\]
where $t_i$ is either $0$ or a positive divisor of $\abs{m}$ and $t_{i+1}$ is the next largest element in the set $\{0\} \cup \{ t \in \Z_{>0} \, ; t \mid \abs{m} \} \cup \{\infty\}$, with the usual ordering for integers, with $0$ smallest,  and with $\infty$ largest.  

Through the embedding  of $\Hp$ into $\HO$, these Weyl-chambers are used to define Weyl-chambers in $\Hp$, on p.\ \pageref{par:defWchHU} and in \eqref{eq:defWtiHU} below. Next, in section \ref{sec:PhimW_and_rho} 
we determine the expansion $\Phi_m^K(Y)$ of the additive lift specific to each Weyl-chamber and from this the Weyl-vector $\rho_m(W)$ of $F_m$.

Next, in section \ref{sec:correction}, we calculate the Weyl-vectors for the $j_{n}$.
Since $j_n(\tau) = F_m(1, \tau) - b_m(0,1)$, for $m = -n$, with a constant $b_m(0,1)$, we need to find the Weyl-vector of a constant. To this aim, we determine the multiplicative Borcherds lift to $\SO(2,2)$ of the constant function $f=1$. It turns out to be
equal to $\eta(z_1)\eta(z_2)$, where $(z_1, z_2)  = Z  \in \HO$. From the product expansion the Weyl-vector attached to $f=1$ comes out as $ \frac{1}{24}(1,1)$.
Also, through pull-back under the embedding $\Hp \hookrightarrow \HO$ from section \ref{sec:embedHp}, 
we immediately get the Borcherds product $\Xi(\tau; 1)$ for $\Ug(L)$ on $\Hp$, which is given by $\Xi(\tau; 1) = \eta(\tau)\eta(-\bar\zeta)$.

Now, for the Weyl-vector attached to $j_n$, $\rho(j_n; W)$, we find 
\[
\rho(j_n; W) = \rho_m(W) - \rho_{0,m}, \quad\text{with}\quad \rho_{0,m} = \sigma({\abs{m}})\cdot(1,1),  
\]
where $m= -n$ and $\sigma(\abs{m})$ denotes the divisor sum $\sum_{d \mid \abs{m}} d$.

With this result, we can once more focus on $\Ug(1,1)$. In section \ref{sec:heegnerU}, we study Heegner divisors:
They are given by locally finite sums of codimension-one sub-manifolds of the symmetric domain, defined through the complement of negative norm lattice vectors. These $\Ug(L)$-invariant divisors are preimages under the canonical projection of divisors on the modular variety $\Ug(L) \backslash \Hp$. 
Contrastingly to higher dimensions, the Heegner divisors on $\Hp$ consist only of discrete points. For a lattice vector $\lambda$ of norm $m$, $m<0$,
a primitive Heegner-point $\tau_\lambda$ of index $m$ is defined as the point in $\Hp$ with the attached representative $z(\tau_\lambda)$ in the complement of $\lambda$.
Note that $\tau_\lambda$ is $\F$-rational and a CM-point, the CM-order of which can be determined explicitly, its conductor turns out to be $\abs{m}$, or a divisor of $\abs{m}$, which can be determined explicitly.

Finally, section \ref{sec:BoProd} contains our main result: Borcherds products for $j_{n}(\tau)$, $n>0$, and, as a corollary, for arbitrary input functions $f\in\Mweak_0(\Gamma)$. This result follows from the author's more general main theorem 8.1 in \cite{Ho12}, together with the explicit calculation of Weyl-vectors and Heegner divisors. Also, in the present setting, we can write the lifting of $j_n$ as a function $\Xi(\tau; j_n)$ depending only on $\tau \in \Hp$, and neither on the particular identification of $\Hp$ with a subset $\coneU$ of the projective space, nor on the specific choice of basis vectors for the lattice $L$: 
\begin{theorem} \label{thm:mainU11_into}
Let $j_{n} = q^{-n} + \sum_{m>0} c(m)q^m \in \Mweak_0(\Gamma)$. Then,
there is a meromorphic function $\Xi(j_n)$ on the complex upper
half-plane $\Hp$, the Borcherds lift of $j_n$, with the following properties
\begin{enumerate}
 \item $\Xi(\tau; j_n)$ is a meromorphic modular form of weight $0$ for
$\Ug(L)$ with a multiplier system of (at most) finite order.
 \item The zeros and poles of $\Xi(j_n)$ lie along the following divisor
 \[
   \Div\bigl(\Xi(j_n)\bigr) 
 = \frac{1}{2} \sum_{\substack{\lambda\in L \\ \hlf{\lambda}{\lambda} = -n}} [\tau_\lambda],
  \] 
where $[\tau_\lambda]$ denotes the $\Ug(L)$-orbit of the primitive Heegner divisor $\tau_\lambda$.
\item For each Weyl-chamber $W$, $\Xi(j_n)$ has an infinite product
expansion. For \\ $W = W(t_i, t_{i+1})$ it takes the form
\[
\Xi\bigl(\tau; j_n, W(t_i, t_{i+1})\bigr) 
= C e\left( \rho_2\tau - \bar\zeta\rho_1 \right) 
\prod_{\substack{(l,k) \in K \\ n l > -k t_i^2}} \left( 1 - e\bigl( l \tau
- k \bar\zeta \bigr)\right)^{c(kl)},
\]
where the components $\rho_1$ and $\rho_2$ of the Weyl-vector are given by 
\[
\rho_1 = -\sum_{\substack{t \mid n \\  t \geq t_{i+1}} } t, \qquad
\rho_2 = -\sum_{\substack{t \mid \abs{m} \\ 0 < t \leq t_i} } \frac{n}{t}.
\]
The product converges absolutely off the set of poles in an upper half-plane \\ $\abs{\delta}\Im \tau > 2 n$.
\end{enumerate}
\end{theorem}

\section{Non-holomorphic Poincar\'{e} series and a basis for $\Mweak_0(\Gamma)$ }
\label{sec:poincarejn}
 Let $\Gamma = \SL_2(\Z)$. 
The input functions for the Borcherds lift are weakly holomorphic modular forms
of weight zero for $\Gamma$. As usual, we denote the space of such
functions by $\Mweak_0(\Gamma)$. A basis of $\Mweak_0(\Gamma)$ is given by the family of
modular forms with principal part $q^{-n}$ and constant term $0$, $j_n =  q^{-n}
+ \mathrm{O}(q)$, for $n=1, 2, \dotsc$. For example, the modular invariant $j$
is equal to $j_1 + 744$. 

In the present paper, we want to determine the Borcherds lift to $\Ug(1,1)$ for
these basis elements. Our approach here is partly based on the work of Bruinier
in \cite{Br02}, where weakly holomorphic modular forms are represented through
non-holomorphic Poincar\'{e} series of negative weight.

We define non-holomorphic Poincar\'{e} series as in \cite{BY}, section 2: Let $m$ be a
 negative integer. 
Let $I_\nu(z)$ be the usual modified Bessel function. For $s\in\C$ 
and $y\in\R\setminus\{ 0\}$, write                                          
\[
\mathcal{I}_s(y) = \sqrt{\frac{\pi\abs{y}}{2}} I_{s - 1/2}\bigl(\abs{y}\bigr).
\]
Then, for $\tau \in \Hp$, $s\in\C$, the non-holomorphic Poincar\'{e} series of
index $m$, $F_m(\tau,s)$, is given by
\[
F_m(\tau, s) = \sum_{\gamma \in \Gamma_\infty \backslash \Gamma}
\mathcal{I}_s\bigl(2\pi \abs{m} \Im(\gamma\tau) \bigr)
e\left( - \abs{m} \Re(\gamma\tau)\right),
\]
where $\Gamma_\infty = \{\left(\begin{smallmatrix}1 & n \\ 0 &
1\end{smallmatrix}\right);\; n \in \Z\}$, as usual.

Now, $F_m$ relates to $j_n$ for $n = \abs{m}$ as follows (cf.\ \cite{BY},  proposition 2.2, theorem 2.3):
\begin{gather*}
F_m(\tau,1) = j_{\abs{m}} + b_m(0,1), \quad\text{with}\\
b_m(0,1) = 24 \sigma(\abs{m}), 
\end{gather*}
where $\sigma(\abs{m})$ denotes the divisor sum $\sigma(\abs{m}) = \sum_{d \mid
\abs{m}} d$.

Thus, in determining the Weyl-vector for the lift of $j_{\abs{m}}$, with the
methods of \cite{Br02}, we must correct for the contribution of the constant $b_m(0,1)$. The required
correction term is calculated in section \ref{sec:correction} below. 

\begin{remark}\label{rmk:factor12}
 Note that the Poincar\'{e} series $F_m(0,1)$ from \cite{BY} as used here differ from the corresponding Poincar\'{e} 
series in \cite{Br02} by a factor of $\frac12$. 
\end{remark}

\section{The upper half-plane as a symmetric domain for $\Ug(1,1)$}
\label{sec:LandHp}
Let $d$ be a square-free negative integer. Consider the imaginary quadratic
number field  $\F = \Q(\sqrt{d})$.
 Denote by $D_\F$ the discriminant of $\F$, and by $\delta$ its square root,
where by the square root we always mean the principal branch of the complex
square root. 
If $d \equiv 1 \pmod{4}$, we have $D_\F = d$, otherwise, $D_\F = 4d$.
  
Let $\OF$ be the ring of integers in $\F$. We write $\OF$ in the form $\OF = \Z
+ \zeta\Z$, with $\zeta$ given by
\[
\zeta = \begin{cases}
 \frac12 {\delta}  & \text{if}\quad D_\F \equiv 0 \pmod{2} \\
 \frac12 \left( 1 + \delta \right) & \text{if}\quad D_\F \equiv 1 \pmod{2}.
        \end{cases}
\]
Denote by $\DF^{-1}$ the inverse different ideal, the $\Z$-dual of $\OF$ with
respect to the trace $\tr_{\F/\Q}$. As a fractional ideal, $\DF^{-1} =
\delta^{-1}\OF$.

Let $L$ be the $\OF$-module $\OF \oplus \DF^{-1}$, generated by $1$ and
$\delta^{-1}$. On $L$, introduce a non-degenerate indefinite hermitian form
$\hlfempty$,
linear in its left, and conjugate linear in its right argument,
\[
\hlf{\lambda}{\lambda'} = \lambda_1 \overline{\lambda_2'} +  \lambda_2
\overline{\lambda_1'}, \quad\text{with}\quad
\lambda_1, \lambda_1' \in \OF, \lambda_2, \lambda_2' \in \DF^{-1}. 
\]
Denote by $\VF$ the two dimensional hermitian space over $\F$ given by $L
\otimes_{\OF} \F \simeq \F^2$ and by $\VFC$ the two dimensional complex
hermitian space $\VFC = \VF \otimes_\F\C = L \otimes_{\OF} \C$, equipped with
the hermitian forms obtained from $\hlfempty$ by extension of scalars
to $\F$ and $\C$, respectively.

$L$ is a hermitian lattice in $\VF$, i.e.\ an $\OF$ submodule of full rank with
the hermitian form $\hlfempty$.
It is an integer lattice, since $\hlf{\lambda}{\mu} \in \DF^{-1}$ for all $\lambda$, $\mu \in L$.
Also, $L$ is even, since $\hlf{\lambda}{\lambda} \in \Z$ for all
$\lambda \in L$. Furthermore, $L$ is unimodular: The dual lattice of $L$, is defined
as
\[
L' = \left\{ v \in \VF\,;\, \hlf{\lambda}{v} \in \DF^{-1}, \quad\text{for
all}\quad \lambda \in L\right\}.
\] 
Clearly, $L = L'$. 

The unitary group $\Ug(V)$ is the group of endomorphisms of $\VF$ preserving the
hermitian form; its set of real points, $\Ug(V)(\R)$,  is the unitary group of
$\VFC$, isomorphic to $\Ug(1,1)$. The group of isometries of $L$ is an
arithmetic subgroup denoted $\Ug(L)$.  

%
Denote by $\PFR \simeq \mathbb{P}^1\C$ the projective space of $\VFC$ and by
$\piU: \VFC \rightarrow \PFR$ the canonical projection. A projective model for
the symmetric domain of $\Ug(V)(\R)$ is given by the following positive cone
\[
\coneU = \left\{  [v]\,;\, \hlf{v}{v} > 0 \right\} \subset \PFR.
\] 
An affine model can be realized as the complex upper half-plane  
$\Hp = \{  \tau\in\C\, ;\, \Im(\tau) > 0 \}$. A biholomorphic map between $\Hp$
and $\coneU$ is constructed as follows (cf.\ \cite{Ho12}, section 2):

Choose an isotropic lattice-vector $\ell \in L$, i.e.\ with $\hlf{\ell}{\ell} =
0$ and a second vector $\ell' \in L$, also isotropic, with $\hlf{\ell}{\ell'}
\neq 0$. The isotropic line $\F\ell$ corresponds to a cusp of the symmetric
domain. 
For each $[v]$ in $\coneU$ we fix a (unique) representative
$z$ of the form 
\begin{equation}\label{eq:normalz}
 z(\tau) = \ell' - \tau\delta\hlf{\ell'}{\ell} \ell \in \piU^{-1}(\coneU).
\end{equation}
Then, by the positivity condition $\hlf{z}{z}>0$, we have $\Im\tau > 0$, thus
$\tau \in \Hp$.
 Conversely, for each $\tau \in \Hp$ there is a $z=z(\tau)$ of the above form with
$\piU(\C z) \in \coneU$. 
The unitary group $\Ug(V)(\R)$ operates on $\Hp$ by fractional linear
transformations. With matrix representation for the basis vectors $\ell$,
$\ell'$, this operation is given by
\begin{equation}\label{eq:isomSUSL}
\Ug(V)(\R) \ni 
\begin{pmatrix}
 A & B \\ C & D 
\end{pmatrix}: \tau \rightarrow
 \frac{A \tau - B \epsilon^{-1}}{ C\epsilon \tau + D}, 
\quad\text{with}\;\epsilon = -\delta \hlf{\ell'}{\ell}.
\end{equation}
The special unitary group $\SU(V)(\R)$ is isomorphic to the special linear group
$\SL_2(\R)$ via
\[
\SL_2(\R) \ni 
\begin{pmatrix}
a & b \\ c & d 
\end{pmatrix} \longmapsto 
\begin{pmatrix}
a & b\epsilon \\ c\epsilon^{-1} & d 
\end{pmatrix} \in \SU(V)(\R).
\]
Note that since $\epsilon \in \DF^{-1}$, the arithmetic subgroup  $\SU(L)$ is
isomorphic to $\SL_2(\Z)$.
In particular, if $\hlf{\ell}{\ell'} = \delta^{-1}$, the isomorphism reduces to
the identity. 

\section{The lattice $L$ as a quadratic module}\label{sec:LandHO}

As a $\Z$-module, the lattice $L$, introduced in the previous section, is
isomorphic to $\Z^4$, namely
\[
 L \simeq  \Z \oplus \zeta \Z \oplus \delta^{-1}\Z \oplus \delta^{-1}\zeta \Z
\simeq \Z^4,
\]
with $\delta$ and $\zeta$ as above. We introduce a non-degenerate indefinite
$\Z$-valued bilinear form on $L$, $\blfempty \vcentcolon = \tr_{\F/\Q}
\hlfempty$. Hence, $L$ acquires the structure of a quadratic $\Z$-module, with
the quadratic form
\[
\Qf{\lambda} \vcentcolon = \frac12 \blf{\lambda}{\lambda} =
\hlf{\lambda}{\lambda} \quad\text{for all}\quad\lambda \in L.
\] 
As such, $L$ is even, since $\Qf{\lambda} \in 2\Z$, for all $\lambda \in L$, and
also unimodular. The signature of $L$ is $(2,2)$. 
We decompose $L$ as a direct sum of two hyperbolic planes over $\Z$, by choosing 
lattice vectors $e_1, e_2, e_3, e_4 \in L$, with $e_1$ and $e_3$ primitive, such that
\[
L = \bigl(\Z e_1 \oplus \Z e_2\bigr) \oplus \bigl( \Z e_3 \oplus \Z e_4\bigr)
\]
and satisfying the following properties:
\begin{equation}\label{eq:defe1234}
\begin{aligned}
\blf{e_i}{e_i} &= 0, \quad i=1,\dotsc, 4 \\
\blf{e_i}{e_j} &= 1,  \quad \text{for}\quad \{i,j\} = \{1,2\}, \{3,4\} \\ 
\blf{e_i}{e_j} &= 0,  \quad\text{otherwise.}
\end{aligned}
\end{equation}
(For a particular choice of these vectors see \eqref{eq:mye1234} in section \ref{sec:embedHp} below.)
Consider the rational quadratic space $\VQ = L\otimes_\Z \Q$ and the real
quadratic space $\VQR = \VQ \otimes_\Q \R$, with $\blfempty$ extended in either
case to a bilinear form on the respective space, $\Q$- or $\R$-valued.  
If we consider $\VF = L\otimes_{\OF}\F$ as a vector space over $\Q$, it is
isomorphic to $\VQ$. Similarly, $\VQR$ is the four-dimensional real vector space
underlying the complex space $\VFC$, with $\blfempty = 2\Re\hlfempty$. 

Let $K=L \cap e_1^\perp \cap e_2^\perp$. Then, $K$ is a $\Z$-submodule of $L$.
The subspace $K\otimes_\Z\R$ of $\VQR$ with the
restriction of $\blfempty$ is a Lorentzian quadratic space, with signature $(1,1)$.
 We introduce the following notation: For $x \in K\otimes_\Z\R$ write 
$x = (x_1, x_2) = x_1 e_3 + x_2 e_4$. \label{Kisdefinedhere}

Denote by $\SO(V)$ the special orthogonal group of $\VQ$ and its set of real points as
$\SO(V)(\R) \simeq \SO(2,2)$. 
We recall some basic facts concerning the symmetric domain for the operation of
$\SO(V)(\R)$, for details, see \cite{Br02} pp.\  78f. 
A model for the symmetric domain is given by the Grassmannian of two-dimensional positive
definite subspaces of $\VQR$, which we denote by $\GrO$.

Another, affine, model is the tube domain model $\HO$, which is realized as
follows. Denote by $\VQC$ the complex four-dimensional quadratic space $\VQC =
\VQR \otimes_\R \C$, with $\blfempty$ extended to a non-degenerate, indefinite
$\C$-valued bilinear form. 

The tube domain model $\HO$ is one of the two connected components of the
following subset of $K\otimes_\Z \C$ 
\[
 \{ Z = X + i Y\, ;\, X,Y \in K\otimes_\Z\R, \,\Qf{Y}>0 \} = \HO \cup \overline{\HO}.
\] 
Denote by $\posQuad$ the set of $Y = (y_1, y_2)\in K\otimes_\Z\R$, with $y_1, y_2 >
0$.  
We fix $\HO$ as the component for which $Y \in \posQuad$. 
Thus, $\HO$ is isomorphic to two copies of the complex upper-half plane
\begin{equation}\label{eq:isomHO2Hp}
\HO \simeq \Hp \times \Hp, \quad\left(z_1, z_2 \right)
= z_1 e_3 + z_2 e_4 
\mapsto (z_1, z_2). 
\end{equation}
Denote by $\piO$ the canonical projection $\piO : \VQC \rightarrow \PQC \simeq
\mathbb{P}^2 \C$. The tube domain can be mapped biholomorphically 
to any one of the two connected components of a positive cone in $\PQC$ given by
\begin{equation*}
\left\{ \bigl[ Z_L \bigr]\,;\, Z_L^2 = 0, \blf{Z_L}{\overline{Z_L}} > 0
\right\}.
\end{equation*}
We fix this component and denote it by $\coneO$. 
For each $[Z_L]$, there is a unique representative $Z_L$ of the form
\begin{equation}\label{eq:normalZL}
Z_L = -\Qf{Z}  e_1 + e_2 + Z, \quad\text{with}\quad Z = z_1 e_3 + z_2 e_4 \in
\HO.
\end{equation}
Conversely, to each $Z \in \HO$ we can associate a unique $Z_L \in
\piO^{-1}(\coneO)$.

Finally, every $Z_L$ as in \eqref{eq:normalZL} can be written in the form $Z_L =
X_L + i Y_L$, with $X_L$, $Y_L\in \VQR$ perpendicular vectors of equal norm, with $X_L^2 = Y_L^2 = Y^2>0$. 
The positive definite two-dimensional (real) subspace of $\VQR$ generated by $X_L$ and $Y_L$,
corresponds to a point in the Grassmannian model $\GrO$.

\section{Embedding of $\Hp$ into the tube domain}\label{sec:embedHp}
Since $V_\F$ is both a hermitian space over $\F$ 
and, as a $\Q$-vector space, a quadratic space, with the bilinear form $\blfempty$ obtained from  
the hermitian form $\hlfempty$, the unitary group $\Ug(V)$ embeds into the special orthogonal group $\SO(V)$.
Similarly, $\Ug(L)$ is embedded into $\SO(L)$.

From this, in turn, one obtains an embedding of the respective symmetric domains.
Such an embedding is constructed in \cite{Ho12} section 4, and \cite{Ho11}, ch.\ 3 in a rather more
general setting.
In the present case, where the signature of $\VQ$ as a quadratic space is
$(2,2)$, this embedding is readily described. 

As in \cite{Ho12}, we introduce the following notation: For a complex scalar $\mu
\in \C \setminus \R$, we denote by $\hat \mu$ the endomorphism of $\VQR$ induced
 from scalar multiplication with $\mu$ in the complex hermitian space $\VFC$.

Now, fix a basis for $L\otimes_\Z\Q$ of the form \eqref{eq:defe1234} by
setting
\begin{equation}\label{eq:mye1234}
\begin{aligned}
e_1 = \ell, && e_3 = - \hat \zeta \ell, && e_2 = 
\left(\frac{\zeta}{\delta\hlf{\ell'}{\ell}}\right)\sphat \ell', &&
 e_4 = \left(\frac{1}{\delta\hlf{\ell'}{\ell}}\right)\sphat \ell'.  
\end{aligned}
\end{equation}
With the identification $\HO \simeq \Hp \times \Hp$ from \eqref{eq:isomHO2Hp}, 
we define an embedding of $\Hp$ into $\HO$ as follows 
\begin{equation} \label{eq:tautoZshort}
\begin{aligned} 
\Hp & \longrightarrow & \Hp \times\Hp  &= \HO\\
\tau & \longmapsto & ( \tau, -\bar\zeta)   &= Z,
\end{aligned}
\end{equation}
where $(\tau, -\bar\zeta) = \tau e_3  -\bar\zeta e_4$, with the notation from section \ref{sec:LandHO} above. 
Clearly, this embedding is biholomorphic.

Recall that $\Hp$ is biholomorphically mapped to the projective cone model
$\coneU$ through the association to $\tau$ of a unique $z$ of the form
\eqref{eq:normalz} with $[z] \in \coneU$ and, similarly $\HO$ is
biholomorphically mapped to a component, $\coneO$, of a projective cone, via $Z
\mapsto [Z_L]$. Thus, the embedding of $\Hp$ into $\HO$ induces a biholomorphic
mapping between  the projective cone models $\coneU$ and $\coneO$, 
as illustrated by the following diagram:
\[
\xymatrix{
**[l] [z] \in \coneU \ar@{^(->}[rr] && \coneO \ni [Z_L] \\
 **[l] \C z \subset \piU^{-1}(\coneU) \ar@{^(->}[rr]  \ar@{->}[u]^{\piU} &&
\piO^{-1}(\coneO) \supset \C Z_L\ar@{->}[u]^{\piO}\\
**[l] \tau \in \Hp \ar@{^(->}[rr] \ar@{<->}[u] && {\Hp\times\Hp \simeq \HO} \ni
Z \ar@{<->}[u]}
\]
Explicitly, if $[Z_L]$ is the image of $[z]$, the induced mapping between the representatives $z$ and $Z_L$, normalized according to \eqref{eq:normalz} and \eqref{eq:normalZL}, respectively, is given by
\begin{equation}\label{eq:ztoZL1234}
z = \ell' - \tau \delta\hlf{\ell'}{\ell}\ell \longmapsto 
Z_L = 
\bar{\zeta} \tau e_1 + e_2 + \tau e_3  - \bar{\zeta} e_4. 
\end{equation}
The real and imaginary parts of $Z_L$ may also be written as
\[
 X_L = \left( \frac{1}{ \hlf{\ell'}{\ell} } \right)\sphat z, \quad\text{and}\quad
 Y_L =  - \left( \frac{ i}{2\hlf{\ell'}{\ell} } \right)\sphat z,
\]
see \cite{Ho12}, section 4. In the present case, this is readily
verified from \eqref{eq:tautoZshort} using \eqref{eq:mye1234}.

\section{The additive lifting to to $\SO(K\otimes_\Z\R)$} \label{sec:PhimK}
In this section, we begin to calculate an additive lift to the special orthogonal group 
$\SO(K\otimes_\Z\R)$ of the Lorentzian space $K\otimes_\Z\R$.
Ultimately, this will allow us to determine the Weyl-vector which appears in the  
Borcherds product expansion of the multiplicative lift.

The symmetric domain of $\SO(K\otimes_\Z\R)$ is most readily described as a
hyperboloid model, cf.\ \cite{Br02}, p.\ 66:
\[
 \mathcal{C}_K = \{ v \in K\otimes_\Z\R; \, v^2 = 1, \blf{v}{e} > 0 \},
\] 
where $e$ is a fixed primitive isotropic lattice vector from $K$. We set $e=e_3$. 

Now, for $Z  = X + i Y \in \HO$, the imaginary part $Y$ is contained in
$K\otimes_\Z\R$ and satisfies $Y^2 = y_1 \cdot y_2 >0$, and further, by
construction, $y_1$, $y_2> 0$. Thus, $\frac{Y}{\abs{Y}}$ corresponds to a point
$v$ in the hyperboloid model. 

Denote by $\Phi^K_m(Y)$ the additive lift to $\SO(K\otimes_\Z\R)$ of the non-holomorphic Poincar\'{e} series
$F_m$. 
It is given by the following expression, according to \cite{Br02}, p.\ 56, 
where $(b_+, b_-) = (1,1)$ is the signature of $K$, 
see remark \ref{rmk:factor12} for the additional factor of $\frac12$: 
\begin{equation}\label{eq:int_complicated}
\begin{split}
\Phi_{m}^K(v,s) & = \frac12
\frac{2(4\pi\abs{m})^{-k/2}}{\Gamma(2s)}\;\times \\
&  \sum_{\substack{\lambda \in \beta + K \\ \Qf{\lambda} = m}}
\int_{0}^{\infty} {\mathbf{M}}_{-k/2, s-1/2} \bigl( 4\pi \abs{m}y \bigr)
y^{b^-/4 +
  b^+/4 -2} \exp\bigl(-4\pi y \Qf{\lambda_v} + 2\pi y m\bigr)\, dy,
\end{split}
\end{equation}
where $\mathbf{M}_{\nu,\mu}$ denotes the $M$-Whittaker function, as defined in
\cite{AbSt}, p.\ 190, and $\lambda_v$ denotes the projection of $\lambda$ 
to $\R v$ for $v\in\mathcal{C}_K$.

In the present case, the signature of $K\otimes_\Z\R$ is $(1,1)$, hence $b_+ =
b_- = 1$, $l = 2$ and $k = 1 -l/2 = 0$. The regularized  lifting is obtained by
evaluating \eqref{eq:int_complicated} at $s_0 = 1 - k/2 = 1$. 
Since $K$ is unimodular, summation in \eqref{eq:int_complicated} extends over
all $\lambda \in K$, with $\Qf{\lambda} = m$.

Hence, the integral takes a much simpler form, which can be evaluated directly: 
\begin{multline}\label{eq:int_simple}
\int_0^\infty \mathbf{M}_{0,1/2}(4\pi\abs{m}y)y^{-3/2} 
e^{-4\pi y \QfNop(\lambda_v) -   2\pi y \abs{m}} dy 
= 4\pi \left(\sqrt{\Qf{\lambda_v} + \abs{m} } -
\sqrt{\Qf{\lambda_v}}\right).
\end{multline}
\begin{remark}
 In fact, in \cite{Br02}, p.\ 56 the integral 
is evaluated using a formula for the Laplace transform of the Whittaker function from
\cite{erdelyi2}, p.\ 215.
Unfortunately, the formula given there is wrong for the particular combination of 
parameters needed here, so that the result in \cite{Br02}, while correct in the 
other cases, does not apply if $(b_+,b_-) = (1,1)$.
\end{remark}
Now, the expression for $\Phi_m^K$ can easily be rewritten in terms of $Y$, with
$v = \frac{Y}{\abs{Y}}$. 
For the lattice vectors $\lambda$ occurring in the sum, the projection to $\R v$
is given by
\[
\lambda_v = \frac{\blf{\lambda}{Y} Y}{ Y^2} \quad\text{and}\quad
 \sqrt{\QfNop(\lambda_v)} = \frac{\abs{\blf{\lambda}{Y}}}{\sqrt{2} \abs{Y}}.
\]
Also, we have 
\[
\Qf{\lambda} = \Qf{\lambda_v} + \Qf{\lambda_{v^\perp}} = m <0, 
\]
where $\lambda_{v^\perp}$ denotes the projection to the orthogonal complement
$v^\perp$. Since $K\otimes_\Z\R$ is 2-dimensional, for the complement, it
suffices to choose one vector $Y'$ perpendicular to $Y$.
For $Y = (y_1, y_2)$, we set $Y'\vcentcolon = (y_1, -y_2)$, so that $\abs{Y'} = \abs{Y}$.

Thus, for $\Phi_m^K(Y)$, we have, from \eqref{eq:int_complicated} by way of
\eqref{eq:int_simple}: 
\[
\begin{aligned}
\Phi^K_m(Y) &
= \frac{4\pi}{\sqrt{2}\abs{Y}} \sum_{\substack{\lambda \in K \\ \Qf{\lambda} =
m}} 
 \left( \abs{\blf{\lambda}{Y'}}  - \abs{\blf{\lambda}{Y}}   \right) \\
 & = \frac{2 \sqrt{2}\pi}{\abs{Y}} \sum_{\substack{\lambda \in K \\ \Qf{\lambda}
= m}} 
 \left( \abs{- \lambda_1 y_2 + \lambda_2 y_1 } - \abs{ \lambda_1 y_2 + \lambda_2
y_1 }   \right).
\end{aligned}
\]
Write $\lambda$ in the form $\bigl(t, \frac{m}{t} \bigr)$, with $t$ an integer
dividing $m$. Then, 
\begin{equation}\label{eq:phiK_lambda_abs}
\begin{aligned}
\Phi^K_m(Y) & = \frac{2 \sqrt{2} \pi}{\abs{Y}} \sum_{\substack{t \in \Z \\ t
\mid m }}
\left(\abs{ - t y_2 + \frac{m}{t} y_1 } - \abs{ t y_2 + \frac{m}{t}  y_1 }
\right) \\
& =  \frac{4 \sqrt{2} \pi}{\abs{Y}} \sum_{\substack{t \in \Z, t>0 \\ t \mid m }}
\left(\abs{ - t y_2 + \frac{m}{t} y_1 } - \abs{ t y_2 + \frac{m}{t}  y_1 }
\right),
\end{aligned}
\end{equation}
where, in the second line, only one of the two vectors $\lambda$ and $-\lambda$
occurs in the sum.
\begin{remark}
The convention adopted here, concerning the signs of $\lambda_1$ and $\lambda_2$
differs from that in \cite{Ho11}, chapter 5. Also, there, the additive lift is 
calculated for a different Poincar\'{e} series, see remark \ref{rmk:factor12}.
\end{remark}

\section{Weyl-chambers of $K\otimes_\Z\R$}\label{sec:weylchambers_K}

In the context of Borcherds' theory, Weyl-chambers are connected components of
the symmetric domain, to each of which we can associate a different Borcherds
product expansion, with different Weyl-vectors. In fact, the Weyl-chambers are
related most closely to the negative norm vectors occurring in the expansion
\eqref{eq:phiK_lambda_abs} of the \lq wall-crossing term\rq\ $\Phi_m^K$. 

\paragraph{Weyl-chambers in $\mathcal{C}_K$}
For the following cf.\ \cite{Br02}, ch.\ 3.1. 
Let $\kappa \in K$ be a lattice vector of negative norm. Denote the orthogonal
complement of $\kappa$ in $K\otimes_\Z\R$ by $\kappa^\perp$. It corresponds to
a sub-manifold of $\mathcal{C}_K$ of codimension $1$, and hence to a line in the $Y$-plane.
 This is an example for a \emph{Heegner-divisor}. More generally, see \cite{Br02}, 
Heegner-divisors can be defined as locally finite unions of codimension one sub-manifolds.

A Heegner divisor of index $m$ is a locally finite union of co-dimension one
sub-manifolds of $\mathcal{C}_K$, 
\[
\HeegO(m) = \bigcup_{\substack{\lambda \in K \\ \QfNop(\lambda) =
m}}\lambda^\perp.
\]    
The Weyl-chambers of index $m$ are defined as the connected components of 
\[
\mathcal{C}_K - \HeegO(m).
\] 
Also, finite intersections of Weyl-chambers in $\mathcal{C}_K$ are
again called Weyl-chambers. 

For the Borcherds lift of $j_n$, or, for that matter, Bruinier's
extended Borcherds-lift of $F_m$, it suffices to consider the Weyl-chambers of
index $m=-n$.

 More generally, for an input function $f \in \Mweak_0(\Gamma)$, with principal
part $\sum_{m<0}c(m)q^m$, one considers Weyl-chambers given by the connected
components of 
\[
\mathcal{C}_K - \bigcup_{\substack{ m \in \Z, m<0 \\ c(m) \neq 0}} \HeegO(m).
\]
Under the identification of $\mathcal{C}_K$ with $\posQuad$, the quadrant of the $Y$-plane where $y_1$, $y_2 >0$, 
the Weyl-chambers become subsets of $\posQuad$, bounded by lines perpendicular to lattice vectors in $K$ of negative norm.

\paragraph{Weyl-chambers in $\HO$ and $\Hp$}\label{par:defWchHU}
Let $W$ be a Weyl-chamber in $\mathcal{C}_K \simeq \posQuad$ of index $m$, $m\in \Z_{<0}$.
Then, the set of $Z = X + iY \in \HO$, with $Y$ contained in $W$, 
is called a Weyl-chamber in $\HO$, of index $m$, and also denoted by $W$. 

Finally,  Weyl-chambers in $\Hp$ are defined through the embedding \eqref{eq:tautoZshort} of $\Hp$ into $\HO$: 
 Let $W \subset \HO$ be a Weyl-chamber. 
The set of $\tau$, whose image, $\tau e_3 - \bar{\zeta} e_4$, is
contained in $W$, is called a Weyl-chamber of $\Hp$, also denoted as $W$.
Hence, for $W$ a Weyl-chamber in $\posQuad$, a Weyl-chamber of $\Hp$ is given by  
\begin{equation}\label{eq:defWinHp}
W \vcentcolon = \left\{ \tau = u + iv \in \Hp\,;\, (v, {\textstyle{\frac12}
}\abs{\delta}) \in W  \subset \posQuad\right\}. 
\end{equation}

\paragraph{Explicit description in $\posQuad$} For $m$ a fixed negative integer, we describe
the Weyl-chambers of index $m$ more explicitly. 

For $\lambda \in K$ with $\Qf{\lambda} = m$, write $\lambda = (t, \frac{m}{t})$ with $t\in\Z$, $t \mid m$. 
Since $\lambda^\perp = (-\lambda)^\perp$, it suffices to consider $t>0$.  
The Heegner-divisor $(\pm \lambda)^\perp$ corresponds to the line in $\posQuad$ spanned by $(t, \frac{\abs{m}}{t})$.

Thus, the Weyl-chambers of index $m$ are given by the subsets of $\posQuad$ bounded by such lines, 
as $t$ runs through the positive divisors of $\abs{m}$, and by the boundary components of $\posQuad$, the $y_1$- and $y_2$-axis.
Denote by $d(\abs{m})$ the number of positive divisors of $\abs{m}$ and arrange these divisors by size:
\[
1 = t_1 \leq t_2 \leq \dotsb \leq t_{d(\abs{m})} = \abs{m}.
\]  
Formally, include $t_0 = 0$ and $\infty$, with $0 < t_i$ and $t_i<\infty$, for all $1\leq i \leq d(\abs{m})$.
Then, the Weyl-chambers of index $m$ are given by precisely the following subsets of $\posQuad$:
\begin{equation} \label{eq:def_WtiHO}
\begin{gathered}
W(t_i, t_{i+1}) \vcentcolon = \left\{ Y \in \posQuad;\,
 \frac{t_i^2}{\abs{m}} y_2 < y_1 < \frac{t_{i+1}^2}{\abs{m}} y_2 \right\}, 
\quad\text{for $i=1, \dotsc, d(m)$,} \\
W(0,1) \vcentcolon  = \{ Y \in \posQuad;\, y_1 < \abs{m}^{-1} y_2 \},\qquad
W(\abs{m}, \infty) \vcentcolon = \{ Y \in \posQuad;\, \abs{m} y_2 < y_1 \},
\end{gathered}
\end{equation}
where $W(0,1)$ and $W(\abs{m}, \infty)$ are bounded by the Heegner divisor $(1, \abs{m})$ 
and the $y_1$-axis, and by $(\abs{m},1)$ and the $y_2$-axis, respectively.
Note that passing through increasing $t_i$, the Weyl-chambers are traversed in
clockwise orientation.

\begin{remark}
In $\posQuad$ there are $d(\abs{m}) + 1$ Weyl-chambers of index $\abs{m}$.
These come in mirror-symmetric pairs, with, if $\abs{m}$ is not a perfect square,
 the exception of a middle Weyl-chamber containing the diagonal.   
   Reflection along the diagonal takes $W(0,1)$ to $W(\abs{m},
\infty)$ and, generally, $W(t_i, t_{i+1})$ to $W(t_{d(\abs{m}) - i},
t_{d(\abs{m})+1 - i})$. If $\abs{m}$ is a square, the diagonal is given by a
Heegner-divisor as $\sqrt{\lvert m\rvert } \in \Z_{>0}$. Otherwise, the diagonal is
contained in a \lq middle\rq\ Weyl-chamber, $W(t_i, t_{i+1})$, with
$i = d(\abs{m})/2$, which is mirror-symmetric to itself.
\end{remark}

\paragraph{Explicit description in $\Hp$} From \eqref{eq:defWinHp}, we see that Weyl-chambers
 are strip-shaped regions of the upper-half plane, bounded by lines parallel to the real axis.
For the Weyl-chamber $W(t_i, t_{i+1})$ in $\posQuad$, the corresponding Weyl chamber in $\Hp$ is given by
\begin{equation}\label{eq:defWtiHU}
W(t_i, t_{i+1}) = \left\{ \tau \in \Hp\,;\;  \frac12 \abs{\delta}
\frac{t_i^2}{\abs{m}}  < \Im \tau < \frac12 \abs{\delta} 
\frac{t_{i+1}^2}{\abs{m}}\right\},
\end{equation}
for $i=0,1\dotsc, d(m)$. The \lq topmost\rq\ Weyl-chamber, $W(\abs{m},
\infty)$ is the half-plane given by $2 \Im\tau > {\abs{\delta} \abs{m}}$.

\section{$\Phi^K_m(Y;W)$ and the Weyl-vector $\rho_m(W)$} \label{sec:PhimW_and_rho}
Now, let $W$ be a Weyl-chamber of index $m$ and $\lambda
\in K$ a lattice vector with norm $m$.  Then, $\blf{\lambda}{Y}$ is non-zero for all $Y \in W$, 
and thus has constant sign on $W$.
We introduce the following notation: If $\blf{\lambda}{Y} >0$ for all $Y \in W$,
we write $\blf{\lambda}{W}>0$ and, similarly, $\blf{\lambda}{W}<0$ if $\blf{\lambda}{Y}<0$ for all $Y\in W$. 

Since $Y \in \posQuad$, and thus $y_1, y_2 >0$, the expression from \eqref{eq:phiK_lambda_abs} takes
the form
\begin{align}
\Phi_m^K(Y;W) & = \frac{4 \sqrt{2} \pi }{\abs{Y}}
\Bigl( \sum_{\substack{ t \in \Z_{>0} \\ t \mid m \\ \blfp{ (t, m/t), W }>0}}
 \left(  t y_2 - \frac{m}{t} y_1  - t y_2 -  \frac{m}{t}  y_1 \right) \notag \\
& + \sum_{\substack{ t \in \Z_{>0} \\ t \mid m \\ \blfp{ ( t, m/t), W }<0}}
 \left(  t y_2 - \frac{m}{t} y_1  + t y_2 +  \frac{m}{t}  y_1  \right)\Bigr) \notag\\
&   = \frac{8 \sqrt{2} \pi }{\abs{Y}}
\Bigl( \sum_{\substack{ t \in \Z_{>0} \\ t \mid m \\ \blfp{ (t, m/t), W }>0}} 
 \frac{\abs{m}}{t} y_1  + 
 \sum_{\substack{ t \in \Z_{>0} \\ t \mid m \\ \blfp{ (t, m/t), W }<0}}
  t y_2 \Bigr). \label{eq:PhiKmW_gen}
\end{align}  
Note that, since $m/t <0$ and $Y \in \posQuad$, we have $\blf{\lambda}{Y'} <0$
for all $Y'$.
For Weyl chambers $W(t_i, t_{i+1})$, as introduced in \eqref{eq:def_WtiHO}
above, with $(t_i, t_{i+1})$ either $(0, 1)$, $(\abs{m}, \infty)$ or $t_i$, $t_{i+1}$ two successive divisors of $\abs{m}$, the summation can be made more explicit:
 
For $\lambda \in K$ of the form $(t, \frac{m}{t})$, with $t$ a positive divisor
of $\abs{m}$, we have  
\[
\operatorname{sign} \blf{ \lambda }{ \bigl(t_i, \frac{\abs{m}}{t_i} \bigr) } = 
\begin{cases}
  +1  & \quad\text{if}\quad t > t_i, \\
  \phantom{-}0 & \quad\text{if}\quad t = t_i, \\
  -1  & \quad\text{if}\quad t < t_i. 
\end{cases}
\]
It follows that $\blf{\lambda}{W(t_i, t_{i+1})}>0$ if $t \geq t_{i+1}$, 
whereas $\blf{\lambda}{W(t_i, t_{i+1})}<0$, if $t\leq t_i$. 

In particular, for the \lq outermost\rq\ Weyl-chambers $W(0,1)$ and $W(\abs{m},
\infty)$, the sign is fixed for all $\lambda$ of the above form, i.e.\ 
$\blf{\lambda}{W(0,1)} > 0$ and $\blf{\lambda}{W(\abs{m}, \infty)} <0$.

Thus, for $\Phi_m^K(Y; W)$ we get
\begin{gather*}
\Phi_m^K(Y;W(t_i, t_{i+1})) =  \frac{8 \sqrt{2} \pi }{\abs{Y}} \left(
\sum_{ \substack{t\mid m \\  t \geq t_{i+1} } }  \frac{\abs{m}}{t} y_1 
 +  \sum_{ \substack{t\mid m \\  0 < t \leq t_{i} } }  t y_2 \right). \\
\intertext{In particular, }
\Phi_m^K(Y; W(0,1)) = 
\frac{8\sqrt{2} \pi}{\abs{Y}} \cdot  \sigma_{\abs{m}} y_1, \qquad
\Phi_m^K(Y; W(\abs{m},\infty)) = \frac{8\sqrt{2} \pi}{\abs{Y}} \cdot 
\sigma_{\abs{m}} y_2. 
\end{gather*}

\paragraph{The Weyl-vector} 
The Weyl vector $\rho_m(W)$ attached to a Weyl-chamber $W$ is defined through the following equation (cf.\  
\cite{Bo98}, section 10, or \cite{Br02}, p.\ 69):
\[
\Phi_m^K(Y;W) = 8\sqrt{2}\pi \blf{\frac{Y}{\abs{Y}}}{ \rho_m(W)}.
\]
 Hence, from \eqref{eq:PhiKmW_gen}, we have 
\begin{gather}
\rho_m(W) = 
 \sum_{\substack{ t \in \Z_{>0} \\ t \mid m \\ \blfp{ (t, m/t), W }<0}}
  t e_3  + 
\sum_{\substack{ t \in \Z_{>0} \\ t \mid m \\ \blfp{ (t, m/t), W }>0}} 
 \frac{\abs{m}}{t} e_4.  \nonumber \\
\intertext{In particular, for $W = W(t_i, t_{i+1})$, we get the expression }
\rho_m(W(t_i, t_{i+1})) =   \sum_{ \substack{t\mid m \\  0 < t \leq t_{i} } }  
t e_3  +
 \sum_{ \substack{t\mid m \\  t \geq t_{i+1} } }  \frac{\abs{m}}{t} e_4. \label{eq:rhom_Wti}
\end{gather}

\section{The Weyl-vectors for $j_n$} \label{sec:correction} 

In the previous sections, we have calculated the Weyl-vectors for the lifting of
the non-holomorphic Poincar\'{e} series $F_m$, for $m\in \Z, m<0$. As we want
to determine the Borcherds lift of $j_n$, $n>0$, for the Weyl-vectors, we
 must correct for the fact that for $n=\abs{m}$, $j_n (\tau) = F_m(\tau,1) - b_m(0,1)$. 
The lifting is additive in the Weyl-vector, hence we need only find the
Weyl-vector of the constant $b_m(0,1)$. 
To this aim, we calculate the Borcherds product $\Psi_L(Z;1)$, on $\SO(2,2)$,
for the lift of the constant function $f=1 \in \Mweak_0(\Gamma)$.  By \cite{Bo98}, theorem 13.3, the result is a modular form of weight $\frac12$ with some multiplier
system for a modular group in $\SO(2,2)$. 

Since the function $1 = q^0$ has no principal part, there are no Heegner
divisors and the positive cone $\posQuad$ is the only Weyl-chamber. The only
non-zero coefficient occurring is the constant $c(0) = 1$, and the Borcherds
product takes the form 
\[
\Psi_L(Z,1) = e\left( \blf{\rho_f}{Z}\right) \prod_{\substack{m\in \Z \\ m>0}}
\left(1 - e(m z_1)\right)
 \prod_{\substack{m\in \Z \\ m>0}}\left(1 - e(m z_2)\right),
\]  
since $m$ and $n$ must satisfy $\Qf{\lambda} = mn = 0.$ 

With the usual identification $\HO \simeq \Hp \times \Hp$, the Borcherds product
$\Psi_L(Z,1)$ becomes a modular form of parallel weight $(\frac12, \frac12)$
under the operation of $\SL_2(\Z)\times\SL_2(\Z)$. 

By comparing the corresponding product expansions, it is easily seen that
\[
\Psi_L(z_1, z_2; 1) = \eta(z_1)\cdot \eta(z_2), 
\]
up to, possibly, a constant factor of absolute value 1,
since both functions transform with the same weight.
 It follows that the Weyl-vector term is equal to $e\bigl( \frac{1}{24} (z_1 + z_2)\bigr)$ and the
Weyl-vector is given by
\[
\rho_{f=1}(\posQuad) = \frac{1}{24} \bigl(e_3 + e_4\bigr).
\]
\begin{remark}
Through pull-back under the embedding $\Hp\hookrightarrow \HO$ from \eqref{eq:tautoZshort}, from $\Psi_L(z_1, z_2; 1)$, 
we recover a Borcherds product for the lift $\Xi(\tau;1)$ of $f=1$ on $\Ug(1,1)$, in the form
\[
\Xi(\tau; 1) = \eta(\tau) \cdot \eta(-\bar\zeta). 
\] 
\end{remark}
The contribution due to the constant $b_m(0,1)$ to the Weyl-vector $\rho_m(W)$
of the non-holomorphic Poincar\'{e} series is, simply, 
\[
\rho_{0,m} \vcentcolon = \sigma(\abs{m})\bigl( e_3 + e_4\bigr)  = b_m(0,1) \cdot
\rho_1(\posQuad).
\]
Hence, for $j_{\abs{m}}$, the Weyl-vector attached to a Weyl chamber $W$ is
given by
\[
\rho(j_{\abs{m}}; W) = \rho_m(W) - \rho_{0,m}. 
\]
In particular, for $W = W(t_i, t_{i+1})$, we obtain
\begin{equation}\label{eq:rho_ofjm}
\begin{gathered}
\rho\bigl(j_{\abs{m}}; W(t_i, t_{i+1})\bigr) = - \Bigl( \sum_{\substack{t \mid
\abs{m} \\  t \geq t_{i+1}} } t e_3
 + \sum_{\substack{t \mid \abs{m} \\ 0 < t \leq t_i} } \frac{\abs{m}}{t}
e_4\Bigr), \\
\rho\bigl(j_{\abs{m}}; W(0,1)\bigr) = - \sigma(\abs{m}) e_3, \qquad
\rho\bigl(j_{\abs{m}}; W(\abs{m}, \infty)\bigr) = - \sigma(\abs{m}) e_4. 
\end{gathered}
\end{equation}

\begin{remark}\label{rmk:for_flcm}
 Let $f$ be a weakly holomorphic modular form of weigth $0$,  $f\in\Mweak_0(\Gamma)$, 
with Fourier expansion $f = \sum_{m \gg -\infty} c(m)q^m$. Then, $f$ is a linear combination
of $j_n$ (for $n\in\Z_{>0}$, $c(-n) \neq 0$). 
The Weyl-chambers in $\HO$, and thus in $\Hp$, attached to $f$ can be written in
the form
\begin{equation}\label{eq:decompWf}
W = \bigcap_{\substack{ m < 0 \\ c(m) \neq 0}}  W_m,
\end{equation}
cf.\ \cite{Br02} p.\ 88, with $W_m$ Weyl-chambers of index $m$, as considered in the previous sections. 
 Since $f$ can be represented as a linear combination of either non-holomorphic
Poincar\'{e} series $F_m$ or of the modular functions $j_n$, 
the Weyl-vector for $f$ is given by a sum of Weyl-vectors attached to these
objects. In particular, with $W$ decomposed as in \eqref{eq:decompWf},
\[
\rho(f; W) 
= \sum_{\substack{ m < 0 \\ c(m) \neq 0}} c(m) \rho_m \bigl( W_m \bigr) 
= \sum_{\substack{ m < 0 \\ c(m) \neq 0}} c(m) \rho\bigl( j_{\abs{m}}; W_m\bigr)
+ c(0) \rho_1(\posQuad).
\] 
\end{remark}

\section{Heegner divisors for $\Ug(L)$}\label{sec:heegnerU}
Let $\lambda \in L$ be a lattice vector of negative norm,
$\hlf{\lambda}{\lambda} = m \in \Z_{<0}$. 
The complement of $\lambda$ in the hermitian space $\VFC$ is a subspace of
codimension one. 
Its image under the canonical projection defines a closed analytic subset of
codimension one in $\coneU$,
\[
\HeegU(\lambda) \vcentcolon = \left\{ [z] \in \coneU ; \hlf{z}{\lambda} = 0
\right\}.
\]
Since $\Hp \simeq \coneU$ we can associate to
$\HeegU(\lambda)$ a point $\tau_\lambda \in \Hp$,
defined by
\begin{equation}\label{eq:deftaul}
 \hlf{ z(\tau_\lambda) }{\lambda} = 0,
\quad\text{where, as usual,}\quad
z(\tau_\lambda) = \ell' - \tau_\lambda\delta\hlf{\ell'}{\ell} \ell.
\end{equation}
The $\Ug(L)$-orbit of $\tau_\lambda$ defines an invariant divisor on $\Hp$, the
preimage of a 
Heegner-divisor under the map $\Hp \rightarrow \Ug(L) \backslash \Hp$.
More generally, a Heegner-divisor on $\Hp$ of index $m$ is the
$\Ug(L)$-invariant divisor given by
\begin{equation}\label{eq:defHm}
\HeegU(m) = \sum_{\substack{\kappa \in L \\ \hlf{\kappa}{\kappa} = m }} \left[
\tau_\kappa \right]. 
\end{equation}
Now let $\lambda = \lambda_1 \ell + \lambda_2 \ell'$. From the definition
\eqref{eq:deftaul} of the Heegner-point $\tau_\lambda$ we get
\[
\tau_\lambda  
= \overline{\left(\frac{\lambda_1}{\lambda_2 \epsilon}\right)},
\quad\text{where}\quad \epsilon = -\delta\hlf{\ell}{\ell'},
\] 
as in \eqref{eq:isomSUSL}. Since $\epsilon$, $\lambda_1$ and $\lambda_2$ are
contained in (a fractional ideal of) $\F$, the point $\tau_\lambda$ lies in
$\F$, considered as a subset of $\C$. 
Its minimal equation is given by
\begin{gather}
0 = \tau + \tr\bigl(\tau_\lambda \bigr) \tau + \norm \bigl( \tau_\lambda \bigr)
\tau,  \notag \\
 \text{whence}\quad
0 = \norm(\lambda_2\epsilon)\tau^2 + \tr\bigl( \lambda_1 \bar\lambda_2
\bar\epsilon \bigr) \tau + \norm(\lambda_1). \label{eq:tau_mineq}
\end{gather}
The discriminant of the latter equation is $D = \tr(  \lambda_1 \bar\lambda_2
\bar\epsilon )^2 - 4\norm(\lambda_1)\norm(\lambda_2)$.\\  
Since $m = \Qf{\lambda}$, we have
\[
m = \hlf{\lambda}{\lambda} = \frac{1}{\delta} \left( 
   \lambda_1 \bar\lambda_2 \bar\epsilon - \lambda_2 \epsilon \lambda_1
    \right) =  \frac{2}{\abs{\delta}} \Im\bigl( \lambda_1 \bar\lambda_2
\bar\epsilon\bigr).
\]
Hence, the discriminant $D$ can be expressed in terms of $m$:
\[
D = \tr\bigl(\lambda_1 \bar\lambda_2 \bar\epsilon\bigr)^2 
  - 4 \left(  \Re^2 \bigl(  \lambda_1 \bar\lambda_2 \bar\epsilon  \bigr) 
      + \Im^2 \bigl(  \lambda_1 \bar\lambda_2 \bar\epsilon  \bigr) 
  \right) 
 = - \abs{\delta}^2 m^2  
= m^2 \DF.
\] 
Let $q$ be the $\operatorname{gcd}$ of the coefficients in \eqref{eq:tau_mineq}.
Then, the minimal equation for $\tau_\lambda$ has discriminant $D' = (m/q)^2
\DF$ and the CM-order $\mathcal{O}_{\tau_\lambda}$ of the  Heegner-point $\tau_\lambda$, cf.\
\cite{LangEF} chapter 5, §1,  can be written as a $\Z$-module of the form
\[
\mathcal{O}_{\tau_\lambda} = \Z + \frac{D' + \sqrt{D'}}{2} \Z = \Z +
\frac{\abs{m}}{q} \OF.
\] 
Thus, $\mathcal{O}_{\tau_\lambda}$ has conductor $\abs{m}/q$.

\section{Borcherds Products}\label{sec:BoProd}
The main result from \cite{Ho11}, \cite{Ho12} is the construction of the
Borcherds lift for unitary groups $\Ug(1,n)$. This lifting takes weakly holomorphic
modular forms transforming under a Weil-representation of $\SL_2(\Z)$ to
meromorphic automorphic forms on the symmetric domain of the unitary group. 
These automorphic forms have infinite
product expansions as Borcherds products and take their singularities and poles
along Heegner divisors. 
For $\Ug(1,1)$, they are meromorphic elliptic modular forms with
(possibly) a multiplier system of finite order. Further, in the present setting,
the lattice $L$ is unimodular. Hence, for the lift $\Xi(j_n; z)$, the Borcherds
product expansion from the main theorem of \cite{Ho12} takes the following form
(cf.\ \cite{Ho12}, Corollary 8.1):
\begin{equation}\label{eq:boprod_z}
\Xi(j_n; z) = C e\left( \frac{\hlf{z}{\rho(j_n; W)}}{\hlf{\ell'}{\ell}}\right)
\prod_{\substack{ \lambda \in K \\ \blf{\lambda}{W}>0 }}\left( 1 - e\left(
\frac{\hlf{z}{\lambda}}{\hlf{\ell'}{\ell}}  \right)\right)^{c(\QfNop(\lambda))},
\end{equation} 
where, as usual, $K\simeq\Z^2$ is the Lorentzian $\Z$-sublattice of $L$ introduced
in section \ref{sec:LandHO} and 
$z = \ell' - \tau\delta\hlf{\ell'}{\ell}\ell \in \piU^{-1}(\coneU)$.
 Finally, $C$ is a constant of absolute value $1$.
The above product expansion depends explicitly on the choice of $\ell, \ell' \in L$ 
and on the way $\Hp$ is realized as an affine model of the projective positive cone
$\coneU \subset \PFR$. Below, for theorem \ref{thm:mainU11}, we will rewrite 
the product without this dependence. 

First, however, we examine the Weyl-chamber conditions occurring in \eqref{eq:boprod_z}. 
In contrast to the Weyl-chamber conditions considered in section 
 \ref{sec:weylchambers_K}, here, $\blf{\lambda}{W}>0$ is considered for all
$\lambda \in K$ with $\Qf{\lambda} > 0$ and $c\left(\Qf{\lambda}\right) \neq
0$. 
Write $\lambda$ in the form $(l,k) = le_3 + ke_4$. For $\tau \in \Hp$ the image
$Z =(\tau, - \bar\zeta)$  in $\HO = \Hp \times\Hp$ has imaginary part $Y =
(\Im\tau, \frac12 \abs{\delta})$. Thus, the Weyl-chamber condition takes the
form
\begin{equation}\label{eq:WChcond1}
\blf{Y}{\lambda} = k \Im \tau + l \frac{\abs{\delta}}{2} > 0,\quad (\forall \tau
\in \Hp).
\end{equation}
\begin{remark}
 The Weyl-chamber condition can also be defined as in \cite{Ho12}. The resulting
inequality for $\Im\tau$ is the same, as is readily calculated with the
expression for $e_3$ and $e_4$ from \eqref{eq:mye1234}
\begin{equation*}
0 < \Im \frac{\hlf{z}{\lambda}}{\hlf{\ell'}{\ell}} = k \Im\tau + l
\frac{\abs{\delta}}{2}.
\end{equation*}
\end{remark}
Now for $j_n$, the attached Weyl-chambers have index $m= -n$ and we write them in the form
$W(t_i,t_{i+1})$. By definition, for every $\tau \in W(t_i, t_{i+1})$, we have
\[
\frac{\abs{\delta}}{2}\frac{t_i^2}{n} < \Im \tau.
\]
Inserting this into \eqref{eq:WChcond1}, we get the following version of the
Weyl-chamber condition for $\lambda = (l,k)$ and $W=W(t_i, t_{i+1})$
\begin{equation}\label{eq:WChcond2}
 l > -k \frac{t_i^2}{n}, 
\end{equation}
which is, in particular, satisfied if $l$ and $k$ are both positive. 

We are now ready to formulate our main theorem: 
\begin{theorem}\label{thm:mainU11}
 Let $L$ be the even hermitian lattice $\OF\oplus\DF^{-1}$ with a hermitian form of signature $(1,1)$  
and $\Ug(L)$ its group of isometries.
Let $j_n \in \Mweak_0(\Gamma)$ be the unique weakly holomorphic modular form with Fourier expansion 
$q^{-n} + \sum_{m>0} c(m)q^m$. Then, there is a meromorphic function $\Xi(j_n)$ on the complex upper 
half-plane $\Hp$, the Borcherds lift of $j_n$, with the following properties
\begin{enumerate}
 \item $\Xi(j_n)$ is a meromorphic modular form of weight $0$ for
$\Ug(L)$ with a multiplier system of (at most) finite order.
 \item The zeros and poles of $\Xi(j_n)$ lie along the Heegner divisor
$\HeegU(-n)$, as defined in section \ref{sec:heegnerU},
  \[
   \Div\bigl(\Xi(j_n)\bigr) = \frac12 \HeegU(-n) = 
\frac{1}{2}\sum_{\substack{\lambda\in L \\ \hlf{\lambda}{\lambda} = -n}} [\tau_\lambda].
  \] 
\item For each Weyl chamber $W$, $\Xi(j_n)$ has an infinite product
expansion, absolutely convergent off the set of poles in the half-plane
$\abs{\delta}\Im\tau > 2 n$. For $W = W(t_i, t_{i+1})$, as in \eqref{eq:defWtiHU}, 
it takes the
form
\[
\Xi\bigl(\tau; j_n, W(t_i, t_{i+1})\bigr) 
= C e\left( \rho_2\tau - \bar\zeta\rho_1 \right) 
\prod_{\substack{l,k \in \Z \\ n l > -k t_i^2}} \left( 1 - e\bigl( l \tau
- k \bar\zeta \bigr)\right)^{c(kl)},
\]
where $C$ is a constant of absolute value one, 
and the components $\rho_1$ and $\rho_2$ of the Weyl-vector are given by 
\[
\rho_1 = -\sum_{\substack{t \mid n \\  t \geq t_{i+1}} } t, \qquad
\rho_2 = -\sum_{\substack{t \mid n \\ 0 < t \leq t_i} } \frac{n}{t}.
\]
\end{enumerate}
\end{theorem}
\begin{proof}
  Most of the claimed statements now follow by the results of \cite{Ho12}.
\begin{enumerate}
\item In particular, it follows from \cite{Ho12} theorem 8.1, that $\Xi(j_n)$
 exists and is automorphic of weight $0$ for $\Ug(L)$. 
That its multiplier system has at most finite order is implied by Borcherds' result \cite{Bo99Cor},
which can be proved through an embedding trick. 
\item This also follows directly from theorem 8.1 in \cite{Ho12}. In the present setting,
the Heegner divisor of index $m=-n$ is given by a locally finite sum of orbits of Heegner-points
$\tau_\lambda$, see \eqref{eq:defHm} in section \ref{sec:heegnerU} above. 
 \item Consider the Borcherds product expansion from \eqref{eq:boprod_z}. For
$\rho \vcentcolon = \rho(f; W) = (\rho_1, \rho_2)$ and $\lambda = (k,l) \in K$
we have
\[
\frac{\hlf{z}{\rho}}{\hlf{\ell'}{\ell}} 
=  -\bar\zeta \rho_1 + \tau\rho_2, \quad
\frac{\hlf{z}{\lambda}}{\hlf{\ell'}{\ell}}
= l \tau  - k \bar\zeta, 
\]
so the arguments of the exponential function can be written in the claimed form.
 We have already rewritten the Weyl-chamber condition in the more explicit form
\eqref{eq:WChcond1} and, further, for $W = W(t_i, t_{i+1})$ in the form
\eqref{eq:WChcond2}, needed here. So, with the components of the Weyl-vector
 \eqref{eq:rho_ofjm} previously calculated, we get the claimed form of the Borcherds product expansion.
 It remains only to examine the issue of convergence: 
According to \cite{Ho12}, theorem 8.1, the Borcherds product \eqref{eq:boprod_z} converges absolutely
for $z$ in the complement of the set of poles and satisfying $\hlf{z}{z} > 4
\abs{\hlf{\ell'}{\ell}} \abs{m_0}$, where $m_0 = \min\{ m \in \Z;\, c(m)\neq
0\}$, i.e.\ $\abs{m_0} = n$ for $\Xi(j_n)$. This precise
condition, originally for the Borcherds lift to $\Og(2,p)$, is due to Bruinier, see
\cite{Br02} theorem 3.22.  
Hence, since for $\tau \in \Hp$, we have $\hlf{z}{z} = 2\Im\tau\abs{\delta}\abs{\hlf{\ell'}{\ell}}$, 
the product converges absolutely if $\tau$ is not a pole of $\Xi(\tau; j_n, W)$ and satisfies
$\frac12 \Im\tau\abs{\delta} > n$.
\end{enumerate}
\end{proof}
 Since the lifting is multiplicative, see \cite{Ho12} theorem 8.1, by writing a weakly holomorphic modular form 
$f\in \Mweak_0(\Gamma)$ as a linear combination of $j_n$ plus its constant term, the following corollary is immediate.
\begin{corollary}
Let $f \in \Mweak_0(\Gamma)$ be a weakly holomorphic modular form, with Fourier expansion $f = \sum_{m\gg -\infty} c(m)q^m$.
Then there is a meromorphic modular form of weight $c(0)/2$, the Borcherds lift
$\Xi(f)$, transforming with a multiplier system of at most finite order under
$\Ug(L)$, with the following properties
\begin{enumerate}
 \item The zeros and poles of $\Xi(f)$ lie along a Heegner divisor, given by 
\[
\Div\bigl(\Xi(f)\bigr) = \frac{1}{2} \sum_{\substack{m <0 \\ c(m)\neq 0}} c(m) \HeegU(m). 
\]  
\item For each Weyl chamber $W$, $\Xi(f)$  has an infinite product expansion of the
form
\[
\Xi\bigl(\tau; f, W\bigr) 
= C e\left( \rho_2\tau - \bar\zeta\rho_1 \right) 
\prod_{\substack{ k,l \in \Z \\ 2l \Im\tau + k \abs{\delta} > 0\\
\forall \tau \in W}}
 \left( 1 - e\bigl( l \tau - k \bar\zeta \bigr)\right)^{c(kl)},
\]
where $C$ is a constant of absolute value one, and $(\rho_1, \rho_2)$ is the Weyl-vector attached 
to $f$ and $W$, given explicitly in remark \ref{rmk:for_flcm}.
The product converges absolutely off the set of poles for $\Im\tau\abs{\delta} > 2\abs{m_0}$, with
$m_0 = \min\{m_0\in\Z\,;\; c(m) \neq 0\}$.
\end{enumerate}
\end{corollary}
\begin{proof}
 With $f$ written as a linear combination $f = \sum_{\substack{m \in \Z \\ m< 0}} c(m) j_{\abs{m}} + c(0)$, the statements concerning the product expansion and the divisor of $\Xi(f)$ follow from the theorem and its proof by multiplicativity.
Also, by the general result  \cite{Ho12}, theorem 8.1, the Borcherds lift $\Xi(f)$ transforms with weight $c(0)/2$.  
\end{proof}

\section*{Acknowledegements} 
The results in the present paper are mostly generalizations of results from the last chapter of the author's thesis \cite{Ho11}, completed in 2011 at the TU Darmstadt under the supervision of J.\ Bruinier.

\bibliographystyle{plain}
\bibliography{u11_liftv2.bib}

\end{document}